\documentclass[a4paper,11pt]{article}

\usepackage{amssymb,amsthm,amsmath,dsfont}
\usepackage{color,graphicx}
\usepackage{float}
\usepackage{enumerate}
\usepackage[T1]{fontenc}
\usepackage{lmodern}

\newcommand{\R}{{\mathbb{R}}}

\newcommand{\eps}{\varepsilon}
\newcommand{\cC}{\mathcal{C}}

\definecolor{mygreen}{rgb}{0.11,0.7,0.22}
\definecolor{mred}{rgb}{1,0.,0.}
\definecolor{mcyan}{rgb}{1.,0.5,0.}
\definecolor{mgreen}{rgb}{0.,0.5,0.}
\definecolor{mblue}{rgb}{0.,0.,1.}

\newtheorem{theorem}{Theorem}[section]
\newtheorem{prop}{Proposition}[section]

\newtheorem{remark}{Remark}[section]

\def\vertt{\vert\!\vert\!\vert}
\def\pa{\partial}
\def\ds{\displaystyle}

\title{On the existence of solutions for a drift-diffusion system arising in corrosion modelling}
\author{
{\bf I. Lacroix-Violet and C. Chainais-Hillairet} \\[1.mm]
	Laboratoire Paul Painlev\'e, CNRS UMR 8524,\\
	Universit\'e Lille Nord de France,Universit\'e Lille 1,\\ 
	59655 Villeneuve d'Ascq Cedex, France.\\[1.mm]
          Project-Team SIMPAF,\\
  	INRIA Lille Nord de France,\\
	 40 av. Halley, 59650 Villeneuve d'Ascq, France .\\[1.mm]
	\texttt{Ingrid.Violet@math.univ-lille1.fr,}\\
	\texttt{Claire.Chainais@math.univ-lille1.fr}\\\\
}
\begin{document}
\maketitle


\begin{abstract}
In this paper, we consider a drift-diffusion system describing the corrosion of an iron based alloy in a nuclear waste repository. In comparison with the classical drift-diffusion system arising in the modeling of semiconductor devices, the originality of the corrosion model lies in the boundary conditions which are of Robin type and induce an additional coupling between the equations. We prove the existence of a weak solution by passing to the limit on a sequence of approximate solutions given by a semi-discretization in time.
\end{abstract}


\section{Introduction}\label{modele}

\subsection{General framework of the study} 

In this paper, we consider a system of partial differential equations describing the corrosion of an iron based alloy in a nuclear waste repository. At the request of the French nuclear waste management agency ANDRA, investigations are conducted to evaluate the long-term safety assessment of the geological repository of high-level radioactive waste. The concept for the storage under study  is the following: the waste shall be confined in a glass matrix, placed into cylindrical steel canisters and stored in a claystone layer at a depth of several hundred of meters. The long-term safety assessment of the geological repository has to take into account the degradation of the carbon steel used for the waste overpacks, which is mainly caused by generalized corrosion processes. 

In this framework, the Diffusion Poisson Coupled Model (DPCM) has been developed by Bataillon {\em et al}  \cite{Baetal2010} in order to describe the corrosion processes at the surface at the steel canisters. It assumes that the metal is covered by a dense oxide layer which is in contact with the claystones. The model describes the development of the dense oxide layer. In most cases, the thickness of the oxide layer ranges from nanometers to micrometers and is  much smaller than the size of the steel canister, which justifies a 1D-modelling.  The model is made of electromigration-diffusion equations for the transport of charge carriers (three species are considered: electrons, $F\!e^{3+}$ cations and oxygen vacancies), coupled with a Poisson equation for the electrostatic potential. The interaction between the oxide layer and the adjacent layers (metal and claystone) are described by Butler-Volmer laws for the electrochemical surface reactions and potential drops. The model includes moving boundary equations. 

Numerical methods for the approximation of the DPCM model  have been designed and studied  by Bataillon {\em et al } in \cite{BaBoChFuHoTo2011}. Numerical experiments with real-life data shows the ability of model to reproduce expected physical behaviors. However, theoretical questions like existence of solutions and long-time behavior of the DPCM model have not been investigated yet. In this paper, we consider a simplified version of the DPCM model, already introduced in \cite{BaBoChFuHoTo2011}. In this case, the interfaces of the oxide layer are assumed to be fixed and only two charge carriers are taken into account: electrons and cations $F\!e^{3+}$. We will prove the existence of a global in time solution for this simplified model.

\subsection{Presentation of the model} 

We recall here the simplified version of the DPCM model  introduced in \cite{BaBoChFuHoTo2011}. It is a dimensionless model. The unknowns are the densities of electrons and cations, which are respectively denoted by $N$ and $P$ and the electrical potential denoted by $\Psi$. The current densities are denoted by $J_N$ and $J_P$; they contain both a drift part and a diffusion part. As we do not take into account the moving boundary equations in the simplified model, the domain describing the oxyde layer is the interval $(0,1)$. The equations describing the electrochemical behavior of the dense oxyde layer write, for $t\geq 0$: 
\begin{subequations}\label{DDsyst}
 \begin{equation}
\label{DDP}
\partial_t P + \partial_x J_P=0,~~J_P=-\partial_x P -3P\partial_x \Psi, ~~\mbox{in } (0,1),
\end{equation}
\begin{equation}
\label{DDN}
\eps\partial_t N + \partial_x J_N=0,~~J_N=-\partial_x N +N\partial_x \Psi, ~~\mbox{in } (0,1),
\end{equation}
\begin{equation}
\label{poisson}
-\lambda^2 \partial_{xx} \Psi =3P-N+\rho_{hl},~~ \mbox{in } (0,1),
\end{equation}
\end{subequations}
where $\rho_{hl}$  is the net charge density of the ionic species in the host lattice (we will assume that it is a constant in all the sequel) and $\lambda^2$ and $\varepsilon$ are positive dimensionless parameters arising in the scaling. Let us just mention that $\varepsilon$ is the ratio of the mobility coefficients of cations and electrons, then $\varepsilon \ll 1$. 

As the equations for the carrier densities \eqref{DDP} and \eqref{DDN} have the same form, we will use the synthetical form :
\begin{equation}\label{DDu}
\varepsilon_u\partial_t u + \partial_x J_u=0,~~J_u=-\partial_x u -z_u u\partial_x \Psi, ~~\mbox{in } (0,1).
\end{equation}
For $u=P,N$, the charge numbers of the carriers are respectively $z_u=3,-1$ and we respectively have $\varepsilon_u=1,\varepsilon$. 

Let us now focus on the boundary conditions. Charge carriers are created and consumed at both interfaces $x=0$ and $x=1$.  
The kinetics of the electrochemical reactions at the interfaces are described by Butler-Volmer laws at interfaces. It leads to Robin boundary conditions on $N$ and $P$. As in \cite{BaBoChFuHoTo2011}, we assume that the boundary conditions for $P$ and $N$ have exactly the same form. Therefore, for $u=P,N$, they write:
\begin{subequations}\label{CLu}
\begin{align}
-J_u&=r_u^0(u,\Psi) \mbox{ on } x=0, \label{CLu0}\\
J_u&=r_u^1(u,\Psi,V)\mbox{ on } x=1, \label{CLu1}
\end{align}
\end{subequations}
where $V$ is a given applied potential (we just consider here the potentiostatic case) and $r_u^0$ and $r_u^1$ are linear and monotonically increasing functions with respect to their first argument. More precisely, due to the electrochemical reactions at the interfaces, $r_u^0$ and $r_u^1$, for $u=P,N$,  write: 
\begin{subequations}\label{defru}
\begin{align}
r_u^0(s,x)&=\beta_u^0(x)s-\gamma_u^0(x),\label{defru0}\\
r_u^1(s,x,V)&=\beta_u^1(V-x)s-\gamma_u^1(V-x),\label{defru1}
\end{align}
\end{subequations}
where the functions $(\beta_u^i)_{i=0,1}, (\gamma_u^i)_{i=0,1}$ are expressed with the interface kinetic coefficients $(m_u^i, k_u^i)_{i=0,1}$, the positive transfer coefficients $(a_u^i, b_u^i)_{i=0,1}$, the maximum occupancy for octahedral cations in the lattice $P^m$ and the electron density in the state of metal $N^m$. They write, for $u=P,N$: 
\begin{subequations}\label{defbetagamma}
\begin{align}
\beta_u^i(x)&=m_u^i e^{-z_u b_u^i x}+k_u^ie^{z_u a_u^i x},~i=0,1,\label{betau}\\
\gamma_u^0(x)&=m_u^0u^me^{-z_u b_u^0 x},~\gamma_u^1(x)=k_u^1u^me^{z_u a_u^1 x}.\label{gammau}
\end{align}
\end{subequations}

The boundary conditions for the Poisson equation take into account that the metal and the solution can be charged because they are respectively electronic and ionic conductors. Such an accumulation of charges induces a field given by the Gauss law. These accumulations of charges depend on the voltage drop at the interface by the usual Helmholtz law which links the charge to the voltage drop through a capacitance. The parameters $\Delta \Psi_0^{pzc}$ and $\Delta \Psi_1^{pzc}$ are the voltage drop corresponding to no accumulation of charges respectively in the metal and in the solution. Finally, the boundary conditions for the electric potential write:
\begin{subequations}\label{CLpsi}
\begin{equation}
\label{CLpsi0}
\Psi-\alpha_0\partial_x \Psi = \Delta \Psi_0^{pzc},~~\hbox{on}~x=0,
\end{equation}
\begin{equation}
\label{CLpsi1}
\Psi+\alpha_1\partial_x \Psi = V-\Delta \Psi_1^{pzc},~~\hbox{on}~x=1,
\end{equation}
\end{subequations}
where $\alpha_0$ and $\alpha_1$ are positive dimensionless parameters arising in the scaling. 

The system is supplemented with initial conditions, given in $L^{\infty}(0,1)$:
\begin{equation}
\label{CINP}
N(0, x) = N_0(x); ~~P(0, x) = P_0(x),\quad x\in(0,1).
\end{equation}
In all the sequel, we will assume that the interface kinetic coefficients are given constants and satisfy:
\begin{equation}\label{hyp_mk}
m_u^0, k_u^0, m_u^1, k_u^1 >0, \mbox{ for } u=P,N. 
\end{equation}
The positive transfer coefficients satisfy: 
\begin{equation}\label{hyp_ab}
a_u^0, b_u^0, a_u^1, b_u^1 \in [0,1], \mbox{ for } u=P,N, 
\end{equation}
and  we assume that 
\begin{equation}\label{hyp_NPrho}
3P^m-N^m+\rho_{hl}=0.
\end{equation}
Indeed, in the applications, we have $\rho_{hl}=-5$, $P^m=2$ and $N^m=1$, so that the relation \eqref{hyp_NPrho} is satisfied. 
Then, it is expected  that the solution to the corrosion model verifies $0\leq P(t,x)\leq P^m$ and $0\leq N(t,x)\leq N^m$ almost everywhere. These estimates will be proved in Theorem \ref{existence}; they are crucial for the proof of existence of solutions to the corrosion model. 

\subsection{Main result and outline of the paper}

Let us first note that the system of equations \eqref{DDP}-\eqref{poisson} is a simplified case (linear case) of the so-called drift-diffusion system. This model is currently used in the framework of semiconductors device modelling (see for instance \cite{ BookJu, BookMaRiSc1990, BookMa}) or plasma physics (see \cite{Ch1984}). In this context the drift-diffusion model has been widely studied, from the analytical as from the numerical point of view. Let us refer to \cite{DaV96, FaIt95, Ga1985, Ga1994, Ga01, Ju1994, Ju97, JuPe00, JuPe00_2, JuVi2007, BookMaRiSc1990, Sch90} for results on existence and uniqueness of solutions (using some fixed points theorem) and asymptotic limits of the model (as quasi-neutral limit, zero-relaxation time limit, ...). From a numerical point of view, different methods have been proposed and studied (see for instance \cite{BrMaPi89, ChLiPe03, ChPe03, ChCo95, SaSa97, SchGu69}).

In all these papers, the considered boundary conditions are Dirichlet, Neumann or mixed Dirichlet/Neumann boundary conditions (corresponding to the ohmic contacts and the insulated boundary segments of the semiconductor device). Then, the originality of the corrosion model described in this paper lies in the boundary conditions \eqref{CLu}, \eqref{CLpsi} which are of Robin type, and induce an additional coupling between the equations. The authors have shown in  \cite{ChVi2011} the existence of solution for such a model in the stationary case. Moreover in \cite{BaBoChFuHoTo2011}, a numerical scheme for the model \eqref{DDsyst}-\eqref{CINP} is given and analyzed.

Let us now state the main result of this paper: the existence and boundedness of a weak solution to system \eqref{DDsyst}-\eqref{CINP}.

\begin{theorem}\label{existence}
Assume \eqref{hyp_mk}, \eqref{hyp_ab}, \eqref{hyp_NPrho} and that  $N_0,P_0\in L^{\infty}(0,1)$ satisfy 
\begin{equation}\label{CondCINP}
0 \leq P_0(x) \leq P^m,~~0 \leq N_0(x) \leq N^m,~~\mbox{ for almost every } x\in (0,1).
\end{equation}
Let us also assume that $\Delta\Psi_0^{pzc}$ and $\Delta\Psi_1^{pzc}$ verify
\begin{subequations}\label{condpsi}
\begin{equation}
\label{condpsi0}
-\frac{1}{3a_P^0}\left(1+\log\left(a_P^0k_P^0\alpha_0\right)\right) \leq
\Delta\Psi_0^{pzc} \leq
\frac{1}{a_N^0}\left(1+\log\left(a_N^0k_N^0\alpha_0\right)\right), 
\end{equation}
\begin{equation}
\label{condpsi1}
-\frac{1}{b_N^1}\left(1+\log\left(b_N^1m_N^1\alpha_1\right)\right) \leq
\Delta\Psi_1^{pzc} \leq
\frac{1}{3b_P^1}\left(1+\log\left(b_P^1m_P^1\alpha_1\right)\right). 
\end{equation}
\end{subequations}
Then the problem \eqref{DDsyst}-\eqref{CINP} admits a weak solution $(P,N,\Psi)$ which is defined by $P,~N \in L^\infty([0,T]\times (0,1)) \cap L^2(0,T;H^1(0,1))$, $\Psi \in L^2(0,T;H^1(0,1))$ for all $T>0$ and:
\begin{multline}\label{DDufaible}
\mbox{ For } u=N,P, \ \forall {\varphi} \in L^2(0,T;H^1(0,1)),\\
-\varepsilon_u\int_0^T \int_0^1 u \partial_t{\varphi} dxdt +\varepsilon \int_0^1 u_0(x){\varphi}(0,x)dx-\int_0^T \int_0^1 \left(-\partial_x u-z_uu\partial_x\Psi\right)\partial_x {\varphi} dx dt  \\
+ \int_0^T \left[\left(\beta_u^1(V-\Psi(t,1))u(t,1)-\gamma_u^1(V-\Psi(t,1))\right){\varphi}(t,1) \right. \\
+\left. \left(\beta_u^0(\Psi(t,0))u(t,0)-\gamma_u^0(\Psi(t,0))\right){\varphi}(t,0)\right]dt=0, 
\end{multline}
\begin{multline}\label{poissonfaible}
\mbox{ and } \forall {\varphi} \in L^2(0,T;H^1(0,1)),\\
\lambda^2 \int_0^T \int_0^1 \partial_x \Psi \partial_x {\varphi} dx dt - \int_0^T \frac{\lambda^2}{\alpha_1} \left(V-\Psi(t,1)-\Delta \Psi_1^{pzc}\right){\varphi}(t,1)dt \\
+ \int_0^T \frac{\lambda^2}{\alpha_0} \left(\Psi(t,0)-\Delta \Psi_0^{pzc}\right){\varphi}(t,0)dt =\int_0^T \int_0^1 (3P-N+\rho_{hl}){\varphi} dx dt. 
\end{multline}
Moreover the electron and cation densities $N$ and $P$ satisfy for all $t \in [0,T]$ and $x \in (0,1)$
\begin{equation}
\label{BornesNP}
0 \leq P(t,x) \leq P^m,~~0 \leq N(t,x) \leq N^m.
\end{equation}
\end{theorem}

In order to prove this Theorem, we  consider a semi-discretization in time  for \eqref{DDsyst}--\eqref{CINP}, as done by A.~J\"ungel and I. Violet in \cite{JuVi2007} for the quantum drift-diffusion equations. We first show the  existence of a solution to the semi-discrete problem. Then,  we obtain some a priori estimates on the solution to the semi-discrete problem. These estimates will ensure the compactness of a sequence of approximate solutions. Finally by passing to the limit on the time step in the semi-discrete scheme, we obtain the existence of a weak solution for the problem \eqref{DDsyst}-\eqref{CINP} satisfying \eqref{BornesNP}.
\\
The paper is organized as follows. In Section 2, we present the semi-discrete problem associated to \eqref{DDsyst}-\eqref{CINP} and we prove the existence of a solution to this problem. In Section 3, we show some a priori estimates satisfied by the semi-discrete solution. Finally, in Section~4, we establish the convergence of a sequence of approximate solutions and passing  to the limit in the semi-discrete scheme, we prove Theorem \ref{existence}.


\section{The semi discrete problem}

In this Section, we propose a discretization in time of the system \eqref{DDsyst}--\eqref{CINP}. Therefore, we introduce a uniform subdivision $(t_k)_{0\leq k\leq K}$ of the time interval $[0,T]$, for some $T>0$. The time step is denoted by $\Delta t$ and we have $t_k=k\Delta t$. 

We set 
\begin{equation}\label{systSD_init}
P^0=P_0, N^0=N_0.
\end{equation}
 Starting from this initial condition, for all $0\leq k\leq K$, we consider the following system:
\begin{equation}
\label{systSD_Psi}
-\lambda^2 \partial_{xx} \Psi^k=3P^{k}-N^{k}+\rho_{hl}, \mbox{  in (0,1)},
\end{equation}
\begin{equation}
\label{systSD_u}
\frac{\varepsilon_u}{\Delta t} \left(u^{k+1}-u^{k}\right) +\partial_x J_u^k=0,~~J_u^k=-\partial_x u^{k+1}-z_uu^{k+1}\partial_x \Psi^k,\mbox{ in (0,1)}, \mbox{ for } u=P,N,
\end{equation}
supplemented with the following boundary conditions: 
\begin{subequations}\label{systSD_CLpsi}
\begin{equation}
\label{systSD_CLpsi0}
\Psi^k-\alpha_0 \partial_x \Psi^k=\Delta\Psi_0^{pzc}, ~~ \hbox{on}~~x=0,
\end{equation}
\begin{equation}
\label{systSD_CLpsi1}
\Psi^k+\alpha_1 \partial_x \Psi^k=V-\Delta\Psi_1^{pzc}, ~~\hbox{on}~~ x=1,
\end{equation}
\end{subequations}
\begin{subequations}\label{systSD_CLu}
\begin{equation}
\label{systSD_CLu0}
-J_u^k=\beta_u^0(\Psi^k)u^{k+1}-\gamma_u^0(\Psi^k),~~\hbox{on}~~x=0, \mbox{ for } u=P,N,
\end{equation}
\begin{equation}
\label{systSD_CLu1}
J_u^k=\beta_u^1(V-\Psi^k)u^{k+1}-\gamma_u^1(V-\Psi^k),~~\hbox{on}~~x=1, \mbox{ for } u=P,N.
\end{equation}
\end{subequations}

Let us first note that, for each value of $k$, the system \eqref{systSD_Psi}--\eqref{systSD_CLu} is decoupled. Indeed, knowing $P^k$ and $N^k$, we define $\Psi^k$ as the solution to \eqref{systSD_Psi}, \eqref{systSD_CLpsi}. As \eqref{systSD_Psi} is a classical elliptic equation, if the  right-hand side belong to $L^2(0,1)$, $\Psi^k$ will be defined as the solution to a variational formulation (see below) of \eqref{systSD_Psi}, \eqref{systSD_CLpsi}. Therefore, $\Psi^k\in H^1(0,1)$ and as $H^1(0,1) \subset {\mathcal C}^0([0,1])$ the boundary conditions \eqref{systSD_CLu} have a sense. Then, knowing $\Psi^k$, we define $u^{k+1}= P^{k+1}$ or $N^{k+1}$ as the solution to \eqref{systSD_u}, \eqref{systSD_CLu}.  The equation {\eqref{systSD_u}} on $u^{k+1}$ is a linear convection-diffusion equation. Using the Slotboom change of variables \cite{ChVi2011, BookMa} will lead to a classical elliptic equation for which we can also define a variational formulation and the associated weak solution. 

\begin{theorem}\label{existence_discret}
Under the assumptions of Theorem \ref{existence}, the semi-discrete scheme \eqref{systSD_init}--\eqref{systSD_CLu} admits a unique weak solution  $(P^k,N^k,\Psi^k)_{0\leq k\leq K}$ with $\Psi^k \in H^1(0,1)$ for all $0\leq k\leq K$ and $P^k,N^k \in H^1(0,1)$ for all $1\leq k\leq K$.
\end{theorem}

\begin{proof} 
The proof will be done by induction. Let $k\geq 0$, we assume that $P^k, N^k\in L^{\infty}(0,1)$ (a fortiori satisfied if $P^k,N^k \in H^1(0,1)\subset \mathcal{C}^0([0,1])$) are known. This property is satisfied for $k=0$.
Then, the variational formulation associated to the problem on $\Psi^k$ \eqref{systSD_Psi}, \eqref{systSD_CLpsi} writes:
\begin{equation}
\label{poissonfaibleSD}
\lambda^2a_\Psi(\Psi^k,\varphi)=b^k_\Psi(\varphi), \forall \varphi\in H^1(0,1), 
\end{equation}
with
\begin{eqnarray}
&{}&a_\Psi(\Phi,\varphi)=\int_0^1 \partial_x \Phi \partial_x \varphi dx  +\frac{1}{\alpha_1}\Phi (1)\varphi(1)+ \frac{1}{\alpha_0}\Phi (0)\varphi(0),\nonumber \\
&{}&b^k_\Psi(\varphi)=\frac{\lambda^2}{\alpha_1} \left(V-\Delta \Psi_1^{pzc}\right)\varphi(1)+\frac{\lambda^2}{\alpha_0}\Delta \Psi_0^{pzc}\varphi(0) + \int_0^1 (3P^k-N^k+\rho_{hl})\varphi dx.\nonumber
\end{eqnarray}
But, the functional space $H^1(0,1)$ can be equipped either with its usual norm:
$$\|u\|_{H^1}=\int_0^1 (\partial_x u)^2 dx+ \int_0^1 u^2dx,$$
or with the following equivalent norm (see \cite{ChVi2011}): 
$$\vertt u\vertt_{H^1}=\int_0^1 (\partial_x u)^2dx + u(0)^2 + u(1)^2.$$
Indeed, using the compact injection from $H^1(0,1)$ to ${\mathcal C}([0,1])$ and an adaptation of Poincaré inequality, we obtain the existence of $\kappa>0$ such that :
\begin{equation}\label{eq-norms}
\ds\frac{1}{\kappa}\Vert u\Vert_{H^1}\leq \vertt u\vertt_{H^1}\leq \kappa\Vert u\Vert_{H^1} \quad \forall u\in H^1(0,1).
\end{equation}
Using the norm $\vertt\cdot \vertt_{H^1}$, it is immediately clear that $a_{\Psi}(\cdot,\cdot)$ is a bilinear, continuous  and coercive form in $H^1(0,1)$. Moreover, $b_\Psi^k$ is a continuous linear form in $H^1(0,1)$.  Then, the  Lax-Milgram Theorem ensures that the problem \eqref{poissonfaibleSD} admits a unique solution $\Psi^k\in H^1(0,1)$.  Thanks to the embedding from $H^1(0,1)$ to ${\mathcal{C}^0}([0,1])$, we also have  $\Psi^k \in {\mathcal{C}^0}([0,1])$. 

Let us now focus on the problem on $u^{k+1}= P^{k+1}$ or $N^{k+1}$ \eqref{systSD_u}, \eqref{systSD_CLu}. We first use the classical Slotboom change of variables introduced for instance in \cite{BookMa} : 
$u^{k+1}=e^{-z_u\Psi^k}\zeta_u^{k+1}$. The current density {$J_u^k$} rewrites: ${J_u^k} =-e^{-z_u\Psi^k}\partial_x\zeta_u^{k+1}$ and we can consider the variational formulation: 
%
\begin{equation}
\label{LaxMilgramNP}
a_u^k(\zeta_u,\varphi)=b_u^k(\varphi), \forall \varphi\in H^1(0,1), 
\end{equation} 
with 
\begin{eqnarray}
a_u^k(\kappa_u,\varphi)&=&\frac{\varepsilon_u}{\Delta t}\int_0^1 e^{-z_u \Psi^k}\kappa_u\varphi dx + \int_0^1 e^{-z_u \Psi^k}\partial_x\kappa_u\partial_x\varphi dx \nonumber \\
&{}&+ \beta_u^1(V-\Psi^k(1))e^{-z_u \Psi^k(1)}\kappa_u(1)\varphi(1)+ \beta_u^0(\Psi^k(0))e^{-z_u \Psi^k(0)}\kappa_u(0)\varphi(0),\nonumber \\
b_u(\varphi)&=&\frac{\varepsilon_u}{\Delta t}\int_0^1 u^k \varphi dx + \gamma_u^1(V-\Psi^k(1))\varphi(1)+\gamma_u^0(\Psi^k(0))\varphi(0).\nonumber
\end{eqnarray}
As $\Psi^k$ is continuous in $[0,1]$, it is bounded and $e^{-{z_u\Psi^k}}$ is also bounded and its lower bound is strictly positive. The hypothesis \eqref{hyp_mk} also ensures that the functions $\beta_u^0, \beta_u^1, \gamma_u^0, \gamma_u^1$ are strictly positive. Therefore $a_u^k$ is a bilinear, continuous  and coercive form in $H^1(0,1)$ and $b_u^k$ is a continuous linear form in $H^1(0,1)$. {Then},  the Lax-Milgram Theorem answers that the problem \eqref{LaxMilgramNP} admits a unique solution  denoted $\zeta_u^{k+1}\in H^1(0,1)$ {and} $u^{k+1}= e^{-z_u\Psi^k}\zeta_u^{k+1}\in H^1(0,1) $ is the unique weak solution to \eqref{systSD_u}, \eqref{systSD_CLu0}, \eqref{systSD_CLu1}. This ends the proof of Theorem~\ref{existence_discret}.
\end{proof}

Thanks to the semi-discrete scheme \eqref{systSD_init}--\eqref{systSD_CLu} and Theorem \ref{existence_discret}, we can define an approximate solution to the corrosion model \eqref{DDsyst}--\eqref{CINP}. This approximate solution is a piecewise constant function in time. It is denoted by $(P_{\Delta t}, N_{\Delta t}, \Psi_{\Delta t})$ and defined by 
\begin{multline}\label{solapp}
P_{\Delta t}(t,x)=P^{k}(x),\ N_{\Delta t}(t,x)=N^{k}(x),\  \Psi_{\Delta t}(t,x)=\Psi^k(x),\\
\forall x \in (0,1),~\forall t\in (t_{k-1},t_k],\ \forall 1\leq k \leq K. 
\end{multline}

Our objective is to prove the existence of a solution to the corrosion model by passing to the limit on a sequence of approximate solutions $(P_{\Delta t}, N_{\Delta t}, \Psi_{\Delta t})_{\Delta t} $ when $\Delta t$ tends to~0.  Therefore, we need to establish a priori estimates which will yield the compactness of the sequence of approximate solutions. 


\section{A priori estimates}

In this Section we establish the  a priori estimates satisfied by the approximate solution $(P_{\Delta t}, N_{\Delta t}, \Psi_{\Delta t})$: $L^{\infty}$-estimates on $N_{\Delta t}$ and $P_{\Delta t}$, an $L^2(0,T,H^1(0,1))$-estimate on $\Psi_{\Delta t}$ and $L^2(0,T,H^1(0,1))$-estimates on $N_{\Delta t}$ and $P_{\Delta t}$.

\subsection{$L^{\infty}$-estimates on $P_{\Delta t}$, $N_{\Delta t}$}

\begin{prop} \label{prop_Linf}
Let us assume the hypothesis of Theorem \ref{existence} and that the time step $\Delta t$ verifies $\Delta t \leq \tau$ with
\begin{equation}
\label{conddt}
\tau = \lambda^2\min \left(\frac{1}{9P^m},\frac{\varepsilon}{N^m}\right),
\end{equation}
for all $k \in \{0,...,K\}$  we have
\begin{equation}
\label{estlinf}
0 \leq P^{k}(x) \leq P^m,~~ 0 \leq N^{k}(x) \leq N^m ,\quad \forall x\in (0,1),
\end{equation}
and therefore
\begin{equation}
\label{estlinfNDPD}
0 \leq P_{\Delta t} (t,x) \leq P^m,~~ 0 \leq N_{\Delta t} (t,x)\leq N^m, \quad \forall x\in (0,1), \forall t\in (0,T].
\end{equation}

\end{prop}

\begin{proof}
The proof will be done by induction. First, we note that \eqref{estlinf} is satisfied for $k=0$ thanks to \eqref{systSD_init} and  \eqref{CondCINP}. Let us now assume that the property  is satisfied for some $k \in \{0,...,K-1\}$ , we will establish that both inequalities in  \eqref{estlinf} remain true at step $k+1$. Therefore, we adapt the technique used by A. J\"ungel and Y.~J.~Peng in  \cite{JuPe00}. 

Let $u=P,~N$. We first prove the positivity of $u^{k+1}$. Multiplying \eqref{systSD_u} by  $(u^{k+1})^-=\min(u^{k+1},0)$ and integrating over (0,1), we obtain, after an integration by parts: 
\begin{multline}
 \frac{\varepsilon_u}{\Delta t}\int_0^1 \left(\left(u^{k+1}\right)^-\right)^2 dx - \frac{\varepsilon_u}{\Delta t }\int_0^1 u^k\left(u^{k+1}\right)^- dx + \int_0^1 \left(\partial_x \left(u^{k+1}\right)^-\right) ^2dx \\
+\frac{z_u}{2\lambda^2} \int_0^1 \left(\left(u^{k+1}\right)^-\right)^2\left(z_P \left(P^k-P^m\right)+z_N\left(N^k-N^m\right)\right)dx = \left(f_u^1+f_u^0\right), \label{positivite0}
\end{multline}
with 
$$
\begin{aligned}
f_u^0&=J_u^k(0)u^{k+1}(0)^-+\frac{z_u}{2}(u^{k+1}(0)^-)^{2}\pa_x\Psi^k(0),\\
f_u^1&=-J_u^k(1)u^{k+1}(1)^--\frac{z_u}{2}(u^{k+1}(1)^-)^{2}\pa_x\Psi^k(1).
\end{aligned}
$$
But, using the induction hypothesis and the fact that $z_Nz_P\leq 0$, we have:
\begin{multline*}
\frac{z_u}{2\lambda^2} \int_0^1 \left(\left(u^{k+1}\right)^-\right)^2\left(z_P \left(P^k-P^m\right)+z_N\left(N^k-N^m\right)\right)dx\\
\geq \frac{(z_u)^2}{2\lambda^2} \int_0^1 \left(\left(u^{k+1}\right)^-\right)^2\left(u^k-u^m\right) dx,
\end{multline*}
and \eqref{positivite0} implies 
\begin{equation} \label{positivite1}
\frac{1}{\Delta t} \int_0^1 \left(u^{k+1}\right)^{-2}\left(\varepsilon_u-\frac{z_u^2\Delta t}{2\lambda^2}(u^m-u^{k})\right) dx + \int_0^1 \left(\partial_x \left(u^{k+1}\right)^-\right)^2 dx 
\leq (f_u^1+f_u^0).
\end{equation}
We remark that the condition on $\Delta t$ \eqref{conddt} and the induction hypothesis  ensure that 
\begin{equation}\label{positivite2}
\varepsilon_u-\frac{z_u^2\Delta t}{2\lambda^2}(u^m-u^{k}(x)) \geq 0 \quad \forall x\in(0,1).
\end{equation}
Let us now study the sign of $f_u^0$ and $f_u^1$. Therefore,we introduce, as in \cite{BaBoChFuHoTo2011}, the following functions :
\begin{align*}
\xi_u^0(x)&=\gamma_u^0(x)-u^m\beta_u^0(x)+u^m\frac{z_u}{\alpha_0}(x-\Delta\Psi_0^{pzc})\quad\forall x\in\R,\\
\xi_u^1(x)&=\gamma_u^1(x)-u^m\beta_u^1(x)-u^m\frac{z_u}{\alpha_1}(x-\Delta\Psi_1^{pzc})\quad\forall x\in\R.
\end{align*}
As shown in \cite{BaBoChFuHoTo2011}, $\xi_u^0$ and $\xi_u^1$ are nonpositive functions under the hypothesis \eqref{condpsi}.
Using the boundary conditions \eqref{systSD_CLpsi}, \eqref{systSD_CLu}, we have:
\begin{align*}
f_u^0&=\ds\frac{(u^{k+1}(0)^-)^{2}}{2u^m}\left(\xi_u^0(\Psi^k(0))-u^m\beta_u^0(\Psi^k(0))-\gamma_u^0(\Psi^k(0))\right)+\gamma_u^0(\Psi^k(0))u^{k+1}(0)^-\\
f_u^1&=\ds\frac{(u^{k+1}(1)^-)^{2}}{2u^m}\left(\xi_u^1(V-\Psi^k(1))-u^m\beta_u^1(V-\Psi^k(1))-\gamma_u^1(V-\Psi^k(1))\right)\\
&\hspace{9.cm} +\gamma_u^1(V-\Psi^k(1))u^{k+1}(1)^-. 
\end{align*}
As the functions $(\beta_u^i, \gamma_u^i)_{i=0,1}$ are nonnegative on $\R$ and the functions $(\xi_u^i)_{i=0,1}$ are nonpositive on $\R$ for $u=N,P$, we deduce $f_u^0\leq 0$ and $f_u^1\leq 0$. Then, \eqref{positivite1} and \eqref{positivite2} imply that $\left(u^{k+1}\right)^-=0$ on $(0,1)$ and we obtain the positivity of $u^{k+1}$.

Let us now prove that $u^{k+1}\leq u^m$. Therefore, we multiply \eqref{systSD_u} by  $(u^{k+1}-u^m)^+=\max((u^{k+1}-u^m),0)$ and we integrate over $(0,1)$. It yields:
\begin{multline}\label{maximum0}
\frac{\varepsilon_u}{\Delta t}\int_0^1 \left( (u^{k+1}-u^m)^+\right)^2 dx +\frac{\varepsilon_u}{\Delta t}\int_0^1 \left(u^{m}-u^k\right) (u^{k+1}-u^m)^+ dx \\
+\int_0^1\left(\partial_x (u^{k+1}-u^m)^+\right)^2dx+z_u\int_0^1 \partial_x (u^{k+1}-u^m)^+u^{k+1}\pa_x\Psi^k dx\\
+ J_u^k(1)(u^{k+1}(1)-u^m)^- -
J_u^k(0)(u^{k+1}(0)-u^m)^+=0. 
\end{multline} 
But, as $\partial_x (u^{k+1}-u^m)^+u^{k+1}=\frac{1}{2}\pa_x\left(((u^{k+1}-u^m)^{+})^{2}+2u^m (u^{k+1}-u^m)^+\right)$, integrating by parts and using \eqref{systSD_Psi}, we get:
\begin{multline}\label{maximum1}
\int_0^1 \partial_x (u^{k+1}-u^m)^+u^{k+1}\pa_x\Psi^k dx=\\
\frac{1}{2\lambda^2}\int_0^1\left(((u^{k+1}-u^m)^+)^2+2u^m(u^{k+1}-u^m)^+\right)\left(z_P(P^k-P^m)+z_N(N^k-N^m)\right) dx \\
+\frac{1}{2}\left((u^{k+1}(1)-u^m)^{+2}+2u^m(u^{k+1}(1)-u^m)^+\right)\pa_x\Psi^k(1)\\
-\frac{1}{2}\left((u^{k+1}(0)-u^m)^{+2}+2u^m(u^{k+1}(0)-u^m)^+\right)\pa_x\Psi^k(0).
\end{multline}
{Again}, using the induction hypothesis and the fact that $z_Nz_P<0$, we have
\begin{multline*}
\frac{z_u}{2\lambda^2}\int_0^1\left((u^{k+1}-u^m)^{+2}+2u^m(u^{k+1}-u^m)^+\right)\left(z_P(P^k-P^m)+z_N(N^k-N^m)\right) dx\\
\geq \frac{z_u^2}{2\lambda^2}\int_0^1 (u^{k+1}-u^m)^{+2}(u^k-u^m) dx \\
+ \frac{z_u^2u^m}{\lambda^2}\int_0^1 (u^{k+1}-u^m)^+(u^k-u^m) dx.
\end{multline*}
Combined with \eqref{maximum0} and \eqref{maximum1}, it yields:
\begin{multline}\label{maximum2}
\frac{1}{\Delta t} \int_0^1 (u^{k+1}-u^m)^{+2}\left(\varepsilon_u-\frac{z_u^2\Delta t}{2\lambda^2}(u^m-u^{k})\right) dx \\
+\frac{1}{\Delta t}\left(\varepsilon_u-\frac{z_u^2\Delta t u^m}{\lambda^2}\right) \int_0^1 (u^{m}-u^{k})(u^{k+1}-u^m)^+ dx \\
+ \int_0^1 \left(\partial_x \left(u^{k+1}-u^m\right)^+\right)^2 dx \leq g_u^1+g_u^0,
\end{multline}
with
\begin{align*}
g_u^0&=J_u(0)(u^{k+1}(0)-u^m)^++\frac{z_u}{2}\left((u^{k+1}(0)-u^m)^{+2}+2u^m(u^{k+1}(0)-u^m)^+\right)\pa_x\Psi^k(0),\\
g_u^1&=-J_u(1)(u^{k+1}(1)-u^m)^+-\frac{z_u}{2}\left((u^{k+1}(1)-u^m)^{+2}+2u^m(u^{k+1}(1)-u^m)^+\right)\pa_x\Psi^k(1).
\end{align*}
As we did for $f_u^0$ and $f_u^1$, we use the boundary conditions \eqref{systSD_CLpsi} and \eqref{systSD_CLu} and the definition of the functions $(\xi_u^i)_{i=0,1}$ in order to rewrite $g_u^0$ and $g_u^1$ : 
\begin{align*} 
g_u^0&=\ds\frac{(u^{k+1}(0)-u^m)^{+2}}{2u^m}\left(\xi_u^0(\Psi^k(0))-u^m\beta_u^0(\Psi^k(0))-\gamma_u^0(\Psi^k(0))\right)\\
&\hspace{8.cm}+\xi_u^0(\Psi^k(0))(u^{k+1}(0)-u^m)^+,\\
g_u^1&=\ds\frac{(u^{k+1}(1)-u^m)^{+2}}{2u^m}\left(\xi_u^1(V-\Psi^k(1))-u^m\beta_u^1(V-\Psi^k(1))-\gamma_u^1(V-\Psi^k(1))\right)\\
&\hspace{8.cm}+\xi_u^1(V-\Psi^k(1))(u^{k+1}(1)-u^m)^+.
\end{align*}
It is now clear that $g_u^0\leq 0$ and $g_u^1\leq 0$ and therefore, thanks to the hypothesis on the time step \eqref{conddt}, the inequality \eqref{maximum2} implies $(u^{k+1}-u^m)^{+}=0$ 
and $u^{k+1}\leq u^m$ for all $x\in(0,1)$.
It  ends the proof of Proposition \ref{prop_Linf}.
\end{proof}

\subsection{$L^2(0,T,H^1)$-estimate on $\Psi_{\Delta t}$}

\begin{prop} \label{estH1Psiprop}
Under the assumptions of Proposition \ref{prop_Linf}, there exists a constant $C$ {depending only on the data of the problem (but not on $\Delta t$)} such that:
\begin{equation}
\label{estH1Psi}
\vertt \Psi^k \vertt_{H^1((0,1))}  \leq C,\quad \forall 1\leq k\leq K,
\end{equation}
and
\begin{equation}\label{estH1PsiD}
\|\Psi_{\Delta t}\|_{L^2(0,T,H^1(0,1))}\leq C \sqrt{T}.
\end{equation}
\end{prop}

\begin{proof}
Applying \eqref{poissonfaibleSD} with ${\varphi}=\Psi^k$, we get $a_\Psi(\Psi^k,\Psi^k)=b_\Psi^k(\Psi^k)/\lambda^2$. Setting ${\eta}=\min (1,1/\alpha_0,1/\alpha_1)$, we have 
\begin{equation}\label{estimpsi0}
a_\Psi(\Psi^k,\Psi^k)\geq {\eta} \vertt \Psi^k\vertt^2.
\end{equation}
Using Young and Cauchy-Schwarz inequalities, we get 
\begin{multline}\label{estimpsi1}
\ds\frac{1}{\lambda^2}\vert b_\Psi^k(\Psi^k)\vert \leq \frac{{\eta}}{2} \vertt \Psi^k\vertt^2 +\frac{1}{2{\eta}\lambda^2}\int_0^1(3(P^k-P^m)+(N^k-N^m))^2dx\\
+\frac{1}{2{\eta}} (V-\Delta\Psi_1^{pzc})^2+\frac{1}{2{\eta}}(\Delta\Psi_0^{pzc})^2.
\end{multline}
Then, using the $L^{\infty}$-estimate \eqref{estlinf}, we deduce \eqref{estH1Psi} from \eqref{estimpsi0} and \eqref{estimpsi1}.
The obtention of \eqref{estH1PsiD} is straightforward.

\end{proof}

\begin{remark}\label{rem_estPsi}
Let us note that, due to the compact embedding from $H^1(0,1)$ to ${\mathcal C}([0,1])$, the estimate \eqref{estH1Psi} imply that $\Psi^k$ (for all $k$)  is uniformly bounded on the interval $[0,1]$ by a constant depending only on the data of the problem.
\end{remark}

\subsection{$L^2(0,T,H^1)$-estimates on $P_{\Delta t}, N_{\Delta t}$}

Note that using \eqref{estlinfNDPD}, we easily obtain $L^2(0,T;L^2(0,1))$ estimates for $P_{\Delta t}$ and $N_{\Delta t}$ which imply the weak convergence of a sequence of approximate solutions in $L^2(0,T;L^2(0,1))$. But since we have to pass to the limit also in the space derivatives of $P_{\Delta t}$ and $N_{\Delta t}$, it is not sufficient and we need estimates in $L^2(0,T;H^1(0,1))$. They are given in Proposition~\ref{estH1NP}.

\begin{prop} \label{estH1NP}
 Under the assumptions of Proposition \ref{prop_Linf}, there exists a constant $C$, depending only on the data of the problem (but not on $\Delta t$) such that
\begin{equation}
\label{estl2H1}
\|P_{\Delta t} \|_{L^2(0,T;H^1(0,1))} \leq C,~\|N_{\Delta t} \|_{L^2(0,T;H^1(0,1))} \leq C.
\end{equation}
\end{prop}

\begin{proof}

Let us multiply \eqref{systSD_u} by $u^{k+1}$ and integrate over $(0,1)$. 
Since
$$u^{k+1}\left(u^{k+1}-u^k\right) \geq \frac{1}{2}\left(u^{k+1}\right)^2-\frac{1}{2}\left(u^{k}\right)^2,$$
we obtain 
$$
\frac{\varepsilon_u}{2\Delta t} \int_0^1 \left((u^{k+1})^2-(u^k)^2\right)dx+\ds\int_0^1\pa_x {J_u^k} u^{k+1} dx\leq 0.
$$
But, integrating by parts the term $\ds\int_0^1\pa_x {J_u^k} u^{k+1} dx$ and using \eqref{systSD_Psi}, we get 
\begin{multline}\label{estimNP0}
\frac{\varepsilon_u}{2\Delta t} \int_0^1 \left((u^{k+1})^2-(u^k)^2\right)dx+\int_0^1 \left(\pa_x u^{k+1}\right)^2 dx\\
+\frac{z_u}{2\lambda^2}\int_0^1(u^{k+1})^2(z_P(P^k-P^m)+z_N(N^k-N^m))dx \leq {h_u^0+h_u^1},
\end{multline}
with 
\begin{align*}
{h_u^0}&=\frac{(u^{k+1}(0))^2}{2u^m}\left(\xi_u^0(\Psi^k(0))-{u^m}\beta_u^0(\Psi^k(0))-\gamma_u^0(\Psi^k(0))\right)+\gamma_u^0(\Psi^k(0))u^{k+1}(0),\\
{h_u^1}&=\frac{(u^{k+1}(1))^2}{2u^m}\left(\xi_u^1(V-\Psi^k(1))-{u^m}\beta_u^1(V-\Psi^k(1))-\gamma_u^1(V-\Psi^k(1))\right)\\
&\hspace{9.cm}+\gamma_u^1(V-\Psi^k(1))u^{k+1}(1).
\end{align*}
The sign properties of the functions $\xi_u^i$, $\beta_u^i$,$\gamma_u^i$ for $i=0,1$, ensure that 
$$
{h_u^0}\leq 
\gamma_u^0(\Psi^k(0))u^{k+1}(0)\mbox{  and }
{h_u^1}\leq \gamma_u^1(V-\Psi^k(1))u^{k+1}(1).
$$ Then, we can multiply \eqref{estimNP0} by $\Delta t$ and sum it over $k\in\left\{0,\ldots,K-1\right\}$. Using the $L^{\infty}$- estimate \eqref{estlinf}, we get 
\begin{multline}
\| \partial_x u_{\Delta t} \|_{L^2(0,T;L^2(0,1))}\leq \frac{\varepsilon_u}{2}\int_0^1 (u^0)^2(x) dx + \frac{z_u}{2\lambda^2}\sum_{k=0}^{K-1} \Delta t\int_0^1(u^{k+1})^2 (u^m-u^k) dx \\
+ \sum_{k=0}^{K-1} \Delta t\left(\gamma_u^0(\Psi^k(0))u^{k+1}(0) +\gamma_u^1(V-\Psi^k(1))u^{k+1}(1)\right)
\end{multline}
Thanks to \eqref{CondCINP}, \eqref{estlinf}, the continuity of the functions $(\gamma_u^i)_{i=0,1}$ and Remark \ref{rem_estPsi}, we obtain 
$$
\| \partial_x u_{\Delta t} \|_{L^2(0,T;L^2(0,1))}\leq C,
$$
where $C$ depends only on the data of the problem {and not on $\Delta t$}. As the $L^2(0,T;L^2(0,1))$-estimate of $u_{\Delta t}$ is a straightforward consequence of the $L^{\infty}$-estimate \eqref{estlinf}, it concludes the proof of Proposition \ref{estH1NP}.
\end{proof}

\section{Proof of Theorem \ref{existence} }

In this Section,  we first establish some compactness results and the convergence of a sequence of approximate solutions $(N_{\Delta t}, P_{\Delta t},\Psi_{\Delta t})_{\Delta t}$ when $\Delta t$ tends to 0. Then, by passing to the limit in the scheme  \eqref{systSD_init}--\eqref{systSD_CLu}, we prove that the triple $(N,P,\Psi)$ obtained as the limit of the sequence $(N_{\Delta t}, P_{\Delta t},\Psi_{\Delta t})_{\Delta t}$ is a weak solution to \eqref{DDsyst}--\eqref{CINP}.

\subsection{Compactness results} 
The estimates proven in  Propositions \ref{estH1Psiprop} and \ref{estH1NP} imply the weak convergence of a sequence of approximate solutions $(N_{\Delta t})_{\Delta t},~(P_{\Delta t})_{\Delta t}$ and $(\Psi_{\Delta t})_{\Delta t}$ in $L^2(0,T;H^1(0,1))$, when $\Delta t$ tends to zero. But,  the weak convergence of $(N_{\Delta t})_{\Delta t}$ and $(\Psi_{\Delta t})_{\Delta t}$ in $L^2(0,T,H^1(0,1))$ is not sufficient to pass to the limit in the convective terms $\int_0^T\int_0^1 N_{\Delta t} \partial_x \Psi_{\Delta t}\varphi dx $ or in the boundary terms. We need some strong convergence. To prove this result, we will use a compactness argument due to Simon \cite{Si87}. It is based on some estimates on the time translates of the approximate solutions. 
\\
Let us introduce the shift operator. For $u=P,N$ or $\Psi$, we set
$$\left(\sigma_{\Delta t}  u_{\Delta t}\right)(x,t)=u_{\Delta t}(x,t+\Delta t).
$$
It means that 
$
\left(\sigma_{\Delta t}  u_{\Delta t}\right)(x,t)=
u^{k+1}(x)$ for $x \in (0,1)$ and $t \in (t_{k-1},t_k]$.

\begin{prop}\label{est_translatees} Under the assumptions of Proposition \ref{prop_Linf}, there exists a constant $C$ depending only on the data of the problem (and not on $\Delta t$) such that: 
\begin{gather}
\|\sigma_{\Delta t} P_{\Delta t}-P_{\Delta t} \|_{L^2(0,T;L^2(0,1))} \leq C\Delta t,~\|\sigma_{\Delta t} N_{\Delta t} -N_{\Delta t} \|_{L^2(0,T;L^2(0,1))} \leq C \Delta t.\label{estshiftNP}\\
\mbox{ and } \|\sigma_{\Delta t} \Psi_{\Delta t}-\Psi_{\Delta t} \|_{L^2(0,T;L^2(0,1))} \leq C\Delta t.\label{estshiftPsi}
\end{gather}
\end{prop}

\begin{proof}
Using equation \eqref{systSD_u}, Minkowsky inequality and H\"older inequality, we obtain for $u=N,~P$
\begin{eqnarray} 
\|\sigma_{\Delta t} u_{\Delta t}-u_{\Delta t} \|_{L^2(0,T;L^2(0,1))} &\leq& {\frac{\Delta t} {\varepsilon_u}}\left(\|\partial_x\sigma_{\Delta t} u_{\Delta t}\|_{L^2(0,T;L^2(0,1))}  \right. \nonumber \\
&{}&\left.+ |z_u| \|\partial_x \Psi_{\Delta t}\|_{L^2(0,T;L^2(0,1))} \|\sigma_{\Delta t} u_{\Delta t}\|_{L^\infty(0,T;L^\infty(0,1))} \right),\nonumber \\
&\leq&{\frac{\Delta t}{\varepsilon_u}}\left(\| u_{\Delta t}\|_{L^2(0,T;H^1(0,1))}  \right. \nonumber \\
&{}&\left.+ |z_u| \| \Psi_{\Delta t}\|_{L^2(0,T;H^1(0,1))} \| u_{\Delta t}\|_{L^\infty(0,T;L^\infty(0,1))} \right).\label{proofpropsec3_5} 
\end{eqnarray}
Using estimates \eqref{estlinfNDPD}, \eqref{estH1PsiD} and \eqref{estl2H1} in \eqref{proofpropsec3_5}, we obtain \eqref{estshiftNP}. 




By definition $\left(\sigma_{\Delta t} \Psi_{\Delta t}- \Psi_{\Delta t}\right)(t,x)=\Psi^{k+1}(x)-\Psi^k(x)$ for t $\in (t_{k-1},t_k]$. Using \eqref{systSD_Psi}, \eqref{systSD_CLpsi}, we have
\begin{eqnarray}
-\lambda^2 \partial_{xx} \left(\Psi^{k+1}-\Psi^k\right)=3\left(P^{k+1}-P^k\right)-\left(N^{k+1}-N^k\right), ~~x \in (0,1), \label{proofpropsec3_6_1} \\
\left(\Psi^{k+1}-\Psi^k\right)(0)-\alpha_0 \partial_x\left(\Psi^{k+1}-\Psi^k\right)(0)=0, \label{proofprop_sec3_6_2} \\
\left(\Psi^{k+1}-\Psi^k\right)(1)+\alpha_1 \partial_x\left(\Psi^{k+1}-\Psi^k\right)(1)=0. \label{proofprop_sec3_6_3} 
\end{eqnarray}
Multiplying \eqref{proofpropsec3_6_1} {by $\Psi^{k+1}-\Psi^k$} and integrating over {$(0,1)$}, we obtain, {after an} integration by parts:
\begin{eqnarray}
\lambda^2 \int_0^1 \left[\partial_x\left(\Psi^{k+1}-\Psi^k\right)\right]^2 dx +  \frac{\lambda^2}{\alpha_1}\left(\left(\Psi^{k+1}-\Psi^k\right)(1)\right)^2+\frac{\lambda^2}{\alpha_0}\left(\left(\Psi^{k+1}-\Psi^k\right)(0)\right)^2\nonumber \\
= 3 \int_0^1 \left(P^{k+1}-P^k\right)\left(\Psi^{k+1}-\Psi^k\right) dx +\int_0^1 \left(N^{k+1}-N^k\right)\left(\Psi^{k+1}-\Psi^k\right) dx.
\label{proofpropsec3_6_5}
\end{eqnarray}
First, we note that the left hand side in \eqref{proofpropsec3_6_5} is lower bounded by $\frac{\lambda^2\eta}{\kappa} \Vert \Psi^{k+1}-\Psi^k\Vert^2_{H^1}$ with $\eta=\min (1,1/\alpha_0,1/\alpha_1)$ and $\kappa$ defined in \eqref{eq-norms}. Then, using Cauchy-Schwarz and Young inequalities, we obtain the following upper bound for the right hand side in \eqref{proofpropsec3_6_5}: 
$$
\frac{\lambda^2\eta}{2\kappa} \Vert \Psi^{k+1}-\Psi^k\Vert^2_{H^1}+ \frac{9\kappa}{2\lambda^2\eta}
\int_0^1 \left(P^{k+1}-P^k\right)^2 dx+\frac{\kappa}{2\lambda^2\eta}
\int_0^1 \left(N^{k+1}-N^k\right)^2 dx
$$
It yields
$$
\frac{\lambda^2\eta}{2\kappa} \Vert \Psi^{k+1}-\Psi^k\Vert^2_{L^2(0,1)}
\leq \frac{9\kappa}{2\lambda^2\eta}\left(
\int_0^1 \left(P^{k+1}-P^k\right)^2 dx+
\int_0^1 \left(N^{k+1}-N^k\right)^2 dx\right).
$$ 
Summing over $1\leq k\leq K$ and using the estimates \eqref{estshiftNP}, we deduce \eqref{estshiftPsi}. 
It concludes the proof of Proposition \ref{est_translatees}.
\end{proof}

Finally, thanks to the a priori estimates (Propositions \ref{prop_Linf},  \ref{estH1Psiprop} and \ref{estH1NP}) and the estimates on the time translates (Propositions \ref{est_translatees} ), we obtain the convergence of a subsequence of approximate solutions.

\begin{prop}\label{convergence}
Under the assumptions of Theorem \ref{existence}, we consider a sequence of approximate solutions $(P_{\Delta t}, N_{\Delta t}, \Psi_{\Delta t})_{\Delta t}$ defined by the semi-discrete scheme \eqref{systSD_init}--\eqref{systSD_CLu} and \eqref{solapp}. 
Then, there exists $(P,N)\in L^{\infty}([0,T]\times (0,1))\cap L^2(0,T,H^1(0,1))$ and $\Psi\in L^2(0,T,H^1(0,1))$ such that, up to a subsequence, ($u=P, ~N$)
\begin{itemize}
	\item $\Psi_{\Delta t} \rightarrow \Psi$ strongly in $L^2(0,T;\cC([0,1]))$,
	\item $u_{\Delta t} \rightarrow u$ strongly in $L^2(0,T;\cC([0,1]))$,
	\item $\partial_x u_{\Delta t} \rightharpoonup u$ weakly in $L^2(0,T;L^2(0,1))$, $\sigma_{\Delta t} u_{\Delta t} \rightharpoonup u$ weakly in $L^2(0,T;L^2(0,1))$,
	\item $\partial_x\Psi_{\Delta t} \rightharpoonup \partial_x\Psi$ weakly in $L^2(0,T;L^2(0,1))$.
\end{itemize}
\end{prop}

\begin{proof}
Firstly, due to Propositions \ref{estH1Psiprop} and \ref{estH1NP}, the sequences $(\Psi_{\Delta t})_{\Delta t}$ and $(u_{\Delta t})_{\Delta t}$ are bounded in $L^2(0,T;H^1(0,1))$. Then, there exist $\Psi$ and $u$ in $L^2(0,T;H^1(0,1))$ such that $\Psi_{\Delta t}$ tends to $\Psi$ and  $u_{\Delta t}$ tends to $u$ in $L^2(0,T;L^2(0,1))$. Therefore, up to a subsequence, we obtain the weak convergence of $(\partial_x\Psi_{\Delta t})_{\Delta t}$ to  $\partial_x \Psi$ and of $(\partial_x u_{\Delta t})_{\Delta t}$ to  $\partial_x u$ in $L^2(0,T;L^2(0,1))$.


Secondly, to obtain the strong convergence in $L^2(0,T;\cC([0,1]))$ of sequences $(\Psi_{\Delta t})_{\Delta t}$ and $(u_{\Delta t})_{\Delta t}$, we use a compactness argument. More precisely, we use Theorem 5 given by J. Simon in \cite{Si87}. Indeed, we have:
$$ H^1(0,1) \subset \cC([0,1]) \subset L^2(0,1)$$
with a compact injection from $H^1(0,1)$ to $\cC([0,1])$ (due to the one dimension). Moreover due to Propositions \ref{estH1Psiprop}, \ref{estH1NP} and \ref{est_translatees} the sequence $(f_{\Delta t})_{\Delta t}$ (for $f= \Psi,~u$) is bounded in $L^2(0,T;H^1(0,1))$  and 
$$\| \sigma_{\Delta t}f_{\Delta t}-f_{\Delta t}\|_{L^2(0,T;L^2(0,1))} \longrightarrow 0 \hbox{ uniformly when } \Delta t \rightarrow 0.$$
Then the sequence $(f_{\Delta t})_{\Delta t}$ is relatively compact in $L^2(0,T;\cC([0,1]))$, which gives the strong convergence of $(f_{\Delta t})_{\Delta t}$ in $L^2(0,T;\cC([0,1]))$.

Finally let us note that due to proposition \ref{prop_Linf}, sequences $(P_{\Delta t})_{\Delta t}$ and $(N_{\Delta t})_{\Delta t}$ are bounded in $L^\infty([0,T] \times (0,1))$ and we obtain their $\star$-weak convergence in $L^\infty([0,T] \times (0,1))$. Then their limits belong to $L^\infty([0,T] \times (0,1))$.

This completes the proof of proposition \ref{convergence}.

\end{proof}


\subsection{Passage to the limit}

In order to conclude the proof of Theorem \ref{existence}, it remains to prove that $(P,N,\Psi)$ obtained as a limit of the semi-discrete scheme (see Proposition \ref{convergence}) satisfies \eqref{DDufaible} and \eqref{poissonfaible}.
 
Let us first recall the results of the previous sections. Theorem \ref{existence_discret} gives the existence of a unique solution to the problem ($u=N,~P$)
\begin{equation}
\label{conv1}
\frac{\varepsilon_u}{\Delta t}\left(\sigma_{\Delta t}u_{\Delta t}- u_{\Delta t}\right)+\partial_x {J_{u,\Delta t}} = 0,~{J_{u,\Delta t}}=-\partial_x (\sigma_{\Delta t}u_{\Delta t}) - z_u (\sigma_{\Delta t}u_{\Delta t}) \partial_x \Psi_{\Delta t},
\end{equation}
\begin{equation}
\label{conv2}
-\lambda^2 \partial_{xx} \Psi_{\Delta t} = 3  P_{\Delta t} -  N_{\Delta t}+\rho_{hl},
\end{equation}
\begin{equation}
\label{conv3}
-{J_{u,\Delta t}}=\beta_u^0(\Psi_{\Delta t})\sigma_{\Delta t}u_{\Delta t}-\gamma_u^0(\Psi_{\Delta t}),~x=0,
\end{equation}
\begin{equation}
\label{conv4}
{J_{u,\Delta t}}=\beta_u^1(V-\Psi_{\Delta t})\sigma_{\Delta t}u_{\Delta t}-\gamma_u^1(V-\Psi_{\Delta t}),~x=1,
\end{equation}
\begin{equation}
\label{conv5}
\Psi_{\Delta t}-\alpha_0 \partial_x \Psi_{\Delta t}= \Delta \Psi_0^{pzc},~x=0,
\end{equation}
\begin{equation}
\label{conv6}
\Psi_{\Delta t}+\alpha_1 \partial_x \Psi_{\Delta t}= V- \Delta \Psi_1^{pzc},~x=1.
\end{equation}
The variational formulation of \eqref{conv1}-\eqref{conv6} writes : for all ${\varphi} \in L^2(0,T;H^1(0,1))$, 
\begin{multline}\label{SDfaible}
\int_0^T \int_0^1 \frac{\sigma_{\Delta t}u_{\Delta t}- u_{\Delta t}}{\Delta t} {\varphi} dxdt-\int_0^T \int_0^1 \left(-\partial_x \sigma_{\Delta t}u_{\Delta t}-z_u \sigma_{\Delta t}u_{\Delta t}\partial_x\Psi_{\Delta t}\right)\partial_x {\varphi} dx dt  \\
+ \int_0^T \left[\left(\beta_u^1(V-\Psi_{\Delta t}(1))\sigma_{\Delta t}u_{\Delta t}(1)-\gamma_u^1(V-\Psi_{\Delta t}(1))\right){\varphi}(1)\right.\\ 
\left.+\left(\beta_u^0(\Psi_{\Delta t}(0))\sigma_{\Delta t}u_{\Delta t}(0)-\gamma_u^0(\Psi_{\Delta t}(0))\right){\varphi}(0)\right]dt=0, \mbox{ for } u=N,P,
\end{multline} 
\begin{multline}\label{poissonSDfaible}
\lambda^2 \int_0^T \int_0^1 \partial_x \Psi_{\Delta t} \partial_x {\varphi} dx dt - \int_0^T \frac{\lambda^2}{\alpha_1} \left(V-\Psi_{\Delta t}(1)-\Delta \Psi_1^{pzc}\right){\varphi}(1)dt  \\
+ \int_0^T \frac{\lambda^2}{\alpha_0} \left(\Psi_{\Delta t}(0)-\Delta \Psi_0^{pzc}\right){\varphi}(0)dt=\int_0^T \int_0^1 (3P_{\Delta t}-N_{\Delta t}+\rho_{hl}){\varphi} dx dt. 
\end{multline}

Note that
\begin{multline*}
\int_0^T \int_0^1  \frac{\sigma_{\Delta t}u_{\Delta t}-u_{\Delta t}}{\Delta t} {\varphi} dxdt = -\int_0^T \int_0^1 u_{\Delta t}(t,x) \frac{{\varphi}(t,x)-{\varphi}(t-\Delta t,x)}{\Delta t} dx dt \\
-\int_0^1 u_0(x){\varphi}(0,x) dx
\end{multline*}
Then, using proposition \ref{convergence} and passing to the limit $\Delta t$ tends to zero in \eqref{SDfaible}-\eqref{poissonSDfaible} (up to a subsequence) , we get:
\begin{multline*}
-\int_0^T \int_0^1u \partial_t{\varphi} dxdt-\int_0^1 u_0(x){\varphi}(0,x) dx-\int_0^T \int_0^1 \left(-\partial_x u-z_u u\partial_x\Psi\right)\partial_x {\varphi} dx dt  \\
+ \int_0^T \left[\left(\beta_u^1(V-\Psi(1))u(1)-\gamma_u^1(V-\Psi(1))\right){\varphi}(1)\right.  \\ 
\left.+\left(\beta_u^0(\Psi(0))u(0)-\gamma_u^0(\Psi(0))\right){\varphi}(0)\right]dt=0,\mbox{ for } u=N,P,
\end{multline*} 
\begin{multline*}
\lambda^2 \int_0^T \int_0^1 \partial_x \Psi \partial_x {\varphi} dx dt - \int_0^T \frac{\lambda^2}{\alpha_1} \left(V-\Psi(1)-\Delta \Psi_1^{pzc}\right){\varphi}(1)dt   \\
+ \int_0^T \frac{\lambda^2}{\alpha_0} \left(\Psi(0)-\Delta \Psi_0^{pzc}\right){\varphi}(0)dt=\int_0^T \int_0^1 (3P-N+\rho_{hl}){\varphi} dx dt,
\end{multline*}
for all ${\varphi} \in L^2(0,T;H^1(0,1))$. It is the expected result. 

Furthermore, passing to the limit in \eqref{estlinf} gives for all $t \in [0,T]$ and all $x \in (0,1)$
$$0 \leq P(t,x) \leq P^m,~~0 \leq N(t,x) \leq N^m. $$
This ends the proof of Theorem \ref{existence}.


\section{Conclusion}

In this paper, we prove the existence of a weak solution to a simplified corrosion model. The boundedness of the electronic and cationic densities is also established. Let us note that the dimensionless parameter $\varepsilon$ in \eqref{DDN} is in practice very small because it is the ratio of the mobility coefficients of cations and electrons. It could be set to 0 in the model. Then, the question of existence of a solution to the resulting model is still an open problem. Future work will focus on this question.  The full understanding of the corrosion models reduced to a fixed domain will then permit to deal with the DPCM model on a  moving domain (see \cite{BaBoChFuHoTo2011}).

%
\bibliographystyle{plain}
\bibliography{./Biblio}

\end{document}